\documentclass[12pt]{amsart}
\usepackage[T2A]{fontenc}
\usepackage[utf8]{inputenc}   
\usepackage[russian, english]{babel}
\usepackage{enumerate}

\usepackage{float}

\usepackage{hyperref}
\usepackage{amsmath}
\usepackage{amsthm}
\usepackage{amsfonts}
\usepackage{amssymb}
\usepackage{cite}

\usepackage{mathptmx} 
\usepackage[scaled=0.90]{helvet} 
\usepackage{courier} 
\usepackage[T1]{fontenc}
\usepackage{color}

\usepackage[dvips]{graphicx}
\DeclareGraphicsExtensions{.eps}

\newtheorem{theorem}{Theorem}[section]
\newtheorem{utv*}{Proposition}
\newtheorem{hyp*}{Conjecture}
\newtheorem*{example*}{Example}
\newtheorem{lemma}[theorem]{Lemma}
\newtheorem{corollary}[theorem]{Corollary}
\newtheorem{example}[theorem]{Example}

\newtheorem*{th*}{Theorem}

\theoremstyle{definition}

\newtheorem*{rem*}{Remark}
\newtheorem{defin}[theorem]{Definition}

\def\sli{\sum\limits}

\def\g{\gamma}

\def\R{\mathbb{R}}

\def\EE{\mathcal{E}}

\newcommand{\diam}{\operatorname{diam}}

\newcommand{\e}{\varepsilon}
\newcommand{\dist}{\operatorname{dist}}

\def\cyr{\fontencoding{OT2}\fontfamily{wncyr}\selectfont}
\DeclareTextFontCommand{\textcyr}{\cyr}

{\end{list}}

\newcounter{vremennyj}

\def\H{\mathcal{H}}

\author{A. Anderson}
\author{A. Reznikov}
\thanks{The research of A.R.\ was supported, in part, by the National Science Foundation grant DMS-1764398}

\address{Department of Mathematics, Florida State University, Tallahassee, FL 32306}
\email{reznikov@math.fsu.edu}
\address{Department of Mathematics, Florida State University, Tallahassee, FL 32306}
\email{aanderso@math.fsu.edu}

\begin{document}
\title{Minimal Riesz energy on balanced fractal sets} 
\begin{abstract}
We investigate the asymptotic behavior of minimal $N$-point Riesz $s$-energy on fractal sets of non-integer dimension, with algebraically dependent contraction ratios. For $s$ bigger than the dimension of the set $A$, we prove the asymptotic behavior of the minimal $N$-point Riesz $s$-energy of $A$ along explicit subsequences, but we show that the general asymptotic behavior does not exist. 
\end{abstract}
\date{\today}
\maketitle 

\vspace{4mm}

\footnotesize\noindent\textbf{Keywords}: Best-packing points, Cantor sets, Equilibrium configurations, Minimal discrete energy, Renewal theory, Riesz potentials
\vspace{2mm}

\noindent\textbf{Mathematics Subject Classification:} Primary, 31C20, 28A78. Secondary, 52A40

\vspace{2mm}
\section{Introduction}
In this paper we study the {\it discrete Riesz $s$-energy problem} on fractal sets, for $s$ bigger than the dimension of the set. Namely, we fix a compact set $A\subset \R^p$ and for every $s>0$ and every integer $N\geqslant 2$, define
$$
\EE_s(A, N):=\inf_{\omega_N} E_s(\omega_N),
$$
where the infimum is taken over all $N$-point sets $\omega_N=\{x_1, \ldots, x_N\} \subset A$, and
\[
    E_s(\omega_N) := \sum_{i\neq j} | x_i - x_j|^{-s}, \qquad N = 2,3,4,\ldots
\]
We notice that since the kernel
$$
K_s(x,y):=|x-y|^{-s}
$$
is lower semi-continuous, the infimum is always attained. 

When $s\to\infty$, the minimal $s$-energy problem converges (in the sense described in Theorem \ref{thborod}) to the {\it best-packing problem}. Specifically, for a compact set $A$, we define
$$
\delta(A, N):=\sup_{\omega_N} \min_{i\not=j} | x_i -  x_j|,
$$
where the supremum is taken over all $N$-point sets $\omega_N\subset A$. 
The problems of best-packing and minimal energy are getting a lot of attention, see \cite{V1,V2,V3,V4,Serf}. The methods to tackle these problems include geometric measure theory and modular forms; in this paper, our main tool will be a result from probability and ergodic theory.

The minimal energy problem originates from potential theory, where the {\it continuous minimal energy problem} requires one to find a probability measure $\mu$ on $A$ that minimizes
\begin{equation}\label{eq:contenergy}
I_s(A):=\inf_\mu \int K_s(x,y) \textup{d}\mu(x) \textup{d}\mu(y).
\end{equation}
This problem is non-trivial if $I_s(A)<\infty$, which means that $s$ is smaller than the Hausdorff dimension of $A$. In the case when $s$ is bigger than the Hausdorff dimension of $A$, we have $I_s(A)=\infty$; however, if $A$ is an infinite set, $\EE_s(A, N)$ is finite for every $N\geqslant 2$. Therefore, in this case the minimal discrete energy problem is still non-trivial.

In general, asymptotics of the minimal discrete energy arising from pairwise interaction has been the subject of a number of studies \cite{Hardin2005b,Hardin2004, DAMELIN2005845, Brauchart2006}; it has also been considered for random point configurations \cite{Brauchart2015a} and in the context of random processes \cite{Alishahi2015b, Beltran2016b}. The point configurations that minimize the discrete energy find many applications, in particular, to numerical integration. 

When the set $A$ has certain smoothness (e.g. when $A$ is a smooth $d$-dimensional manifold), then it is known, \cite{Hardin2005b}, that for $s>d$ the following limit exists:
\begin{equation}
\label{eq:limittt} \lim_{N\to \infty}\frac{\EE_s(A, N)}{N^{1+s/d}}.
\end{equation}
On the other hand, it was proved in \cite{Borodachov2007} and more explicitly in \cite{reznikov2018riesz} that if $A$ is the classical $1/3$-Cantor set, then the limit \eqref{eq:limittt} does not exist. It was also proved in \cite{reznikov2018riesz} that for the $1/3$-Cantor set, the limit exists along subsequences
$$
\mathcal{N}_\ell:=\{\ell \cdot 2^n, n\in \mathbb{N}\}
$$
for any fixed positive integer $\ell$.

Fractal sets are known to be an important class of sets in geometry and analysis, as they model porous sets. Fractals are also important because they generally lack any rectifiability and, therefore, are good examples of non-smooth sets. Below we define the class of fractals we will be working with.

A pair of sets $ A_1, A_2 $ will be called \textit{metrically separated} if $ | x -  y|\geq \sigma >0  $ whenever $  x\in A_1 $ and $  y \in A_2$. Recall that a \textit{similitude} $ \psi: \mathbb R^p \to \mathbb R^p $ can be written as
\[
    \psi( x) = r O( x) +  z
\]
for an orthogonal matrix $ O \in \mathcal O(p) $, a vector $  z \in\mathbb R^p $, and a contraction ratio $ 0 <r< 1 $.  
The following definition can be found in \cite{Hutchinson}.
\begin{defin}
A compact set $ A\subset \mathbb R^p $ is called a \textit{self-similar fractal} with similitudes $ \left\{\psi_m\right\}_{m=1}^M $ with contraction ratios $ r_m,\, 1\leq m\leq M $ if
\[
    A = \bigcup_{m=1}^M \psi_m(A),
\]
where the union is disjoint.


We say that $A$ satisfies the \textit{open set condition} if there exists a bounded open set $ V\subset \mathbb R^p $ such that 
\[
    \bigcup_{m=1}^M \psi_m(V) \subset V,
\]
where the sets in the union are disjoint. 
\end{defin}
For a self-similar fractal $A$, it is known \cite{MR867284, MR1333890} that its Hausdorff dimension $\dim_H A = d$ where $d$ is such that 
\begin{equation}
    \label{eq:dim}
    \sum_{m=1}^M r_m^d =1.  
\end{equation}
It will further be used that if $ A $ is a self-similar fractal satisfying the open set condition, then $A$ is $d$-regular; meaning,
\begin{equation}
    \label{eq:dregular}
    c^{-1} r^d \leq \H_d(A\cap B( x, r)) \leq c r^d,
\end{equation}
where $B(x,r)$ is an open ball centered at $x$ with radius $r$ and $\H_d$ is the $d$-dimensional Hausdorff measure in $\R^p$.

\section{Best packing on fractal sets}
\label{sec:prior}

Lalley \cite{lalleyacta, lalleyPacking1988} developed a method to tackle the best-packing problem on fractal sets. This method is based on the {\it renewal theorem} that heavily uses probability theory and ergodic theory. To proceed, we need the following definition.
\begin{defin}
We call the numbers $\{r_m\}_{m=1}^M$ {\it independent} if the set
$$
\{k_1 \log r_1 + \ldots + k_M \log r_M \colon k_1, \ldots, k_M \in \mathbb{Z}\}
$$
is dense in $\mathbb{R}$. If this condition is not satisfied; i.e., if for some $h>0$ we have
$$
\{k_1 \log r_1 + \ldots + k_M \log r_M \colon k_1, \ldots, k_M \in \mathbb{Z}\} = h\mathbb{Z},
$$
then we call the numbers $\{r_m\}_{m=1}^M$ {\it dependent}.
\end{defin}
In \cite{lalleyPacking1988}, Lalley applied his methods from \cite{lalleyacta} to prove the following theorem.
\begin{theorem}\label{lalleypack}
Let $A\subset \R^p$ be a fractal set of Hausdorff dimension $d$. If the similitudes $\{\psi_m\}_{m=1}^M$ that define $A$ have independent contraction ratios $\{r_m\}_{m=1}^M$, then the limit
$$
\lim_{N\to \infty} \delta(A, N)N^{1/d}
$$
exists. 
\end{theorem}

To prove this theorem, Lalley used a continuous version of the {\it renewal theorem} (see \cite[p. 363]{feller}). One advantage of our proof is that we will use the discrete version of this theorem stated below.
\begin{theorem}[see section XIII.11 in \cite{feller1} and \cite{lalleydiscrete}]\label{threnew}
Let $\{b_n\}_{n=0}^\infty$ and $\{f_n\}_{n=0}^\infty$ be two sequence of non-negative numbers. Moreover, assume
$$
\sli_{n=0}^\infty f_n = 1, \;\; \mbox{and} \;\; \sli_{n=0}^\infty |b_n|<\infty,
$$
and the set $\{n\colon f_n>0\}$ is not a subset of a proper additive subgroup of $\mathbb{Z}$. If a sequence $\{z_n\}_{n=0}^\infty$ satisfies the {\it renewal equation}
$$
z_n=b_n+\sli_{k=0}^n f_k z_{n-k}
$$
for every $n\geqslant 0$, then the limit 
$$
\lim_{n\to\infty} z_n
$$
exists.
\end{theorem}
Lalley also proved that if the similitudes are dependent, then the limit in Theorem \ref{lalleypack} exists along certain subsequences; however, he did not prove that for dependent similitudes the general limit does not exist. In \cite{Borodachov2007} it was proved that if $r_1=r_2=\ldots=r_M$, then the limit does not exist; in this paper, we prove the general theorem that the limit cannot exist if the similitudes that define $A$ are dependent. We also present an example (Example \ref{example}) of a fractal set with dependent (but not equal) contraction ratios for which we compute the best-packing constant $\delta(A,N)$ for every $N\geqslant 2$.

\section{Relation between minimal energy and best-packing}
In this section we show how the discrete Riesz $s$-energy problem is related to the best-packing problem. We begin with the notation from \cite{Borodachov2007}. Fix a number $d>0$; then, for a compact set $A$ and number $s>0$, denote
\begin{align}
&\underline{g}_s(A):=\liminf_{N\to \infty} \frac{\EE_s(A, N)}{N^{1+s/d}}, && \overline{g}_s(A):=\limsup_{N\to \infty} \frac{\EE_s(A, N)}{N^{1+s/d}}\\
&\underline{g}_\infty(A):=\liminf_{N\to\infty} \delta(A, N)N^{1/d}, && \overline{g}_\infty(A):=\limsup_{N\to\infty} \delta(A, N)N^{1/d}.
\end{align}
The following theorem shows that the best-packing problem can be viewed as the minimal Riesz $s$-energy problem for $s=\infty$.
\begin{theorem}[Proposition 3.1 in \cite{Borodachov2007}]\label{thborod}
Let $A\subset \R^p$ be an infinite compact set, and $d\leqslant p$. Then
$$
\lim_{s\to \infty} \overline{g}_s(A)^{1/s} = \frac{1}{\underline{g}_\infty(A)}, \;\;\; \lim_{s\to \infty} \underline{g}_s(A)^{1/s} = \frac{1}{\overline{g}_\infty(A)}.
$$
\end{theorem}
As a simple corollary, we state the following fact that will be important for us. We leave its proof as an easy exercise for the reader.
\begin{corollary}
Let $A\subset \R^p$ be a compact set, and $0<d\leqslant p$. If the limit
$$
\lim_{N\to \infty}\delta(A, N) N^{1/d}
$$
does not exist, then, for sufficiently large $s$, the limit
$$
\lim_{N\to \infty} \frac{\EE_s(A, N)}{N^{1+s/d}}
$$
also does not exist.
\end{corollary}
\section{Main results}
We state our two main results in this section. The first result shows that, for a fractal set with dependent similitudes, the minimal discrete energy has asymptotics along certain natural subsequences.
\begin{theorem}\label{main1}
Let $A\subset \R^p$ be a fractal set defined by similitudes $\{\psi_m\}_{m=1}^M$ with contraction ratios $\{r_m\}_{m=1}^M$. Assume $r_k = r^{i_k}$ for some $r\in (0,1)$ and positive integers $\{i_k\}_{k=1}^M$ that have no common factor. Let $d$ be the Hausdorff dimension of $A$, and for every $\ell \in \mathbb{N}$ denote $\mathcal{N}_\ell:=\{\lfloor \ell\cdot r^{-nd}\rfloor \colon n\in \mathbb{N}\}$. Then, for $s>d$, the following limit exists:
$$
\lim_{N\in \mathcal{N}_\ell}\frac{\EE_s(A, N)}{N^{1+s/d}}.
$$
\end{theorem}
It is natural to ask if this theorem is sharp. Our next result shows that it is; i.e., if the similitudes that define $A$ are dependent, then the general limit as $N\to \infty$ does not exist. In particular, this result shows that Lalley's assumption from Theorem \ref{lalleypack} is necessary and sufficient for the best-packing problem.
\begin{theorem}\label{main2}
Let $A\subset \R^p$ be a fractal set defined by similitudes $\{\psi_m\}_{m=1}^M$ with contraction ratios $\{r_m\}_{m=1}^M$.  Assume $r_k = r^{i_k}$ for some $r\in (0,1)$ and positive integers $\{i_k\}_{k=1}^M$. Then the limit
$$
\lim_{N\to \infty} \delta(A, N)N^{1/d}
$$
does not exist (as always, $d$ denotes the Hausdorff dimension of $A$). Therefore, for sufficiently large $s$, the limit
$$
\lim_{N\to \infty}\frac{\EE_s(A, N)}{N^{1+s/d}}
$$
does not exist.
\end{theorem}
\begin{example}\label{example}
Take the set $A$ defined by similitudes 
$$
\psi_1(x)=x/4, \;\;\; \psi_2(x)=x/2+1/2
$$
acting on $[0,1]$.
Let $F_n$ be the $n$'th Fibonacci number such that $F_1=F_2=1$, and $F_n=F_{n-1}+F_{n-2}$ for $n\geqslant 3$. For $n\geqslant 3$, we can prove that
$$
\delta(A, F_n) = 2^{3-n}.
$$
Moreover, for every integer $N\in (F_{n-1}, F_{n}]$ we have
$$
\delta(A, N) = 2^{3-n},
$$
and so the limit 
$$
\lim_{N\to\infty} \delta(A, N)N^{1/d}
$$
does not exist. We will carry out the details in Section \ref{Sec:example}.
\end{example}
\section{Proof of Theorem \ref{main1}}
To prove Theorem \ref{main1}, we will need the following lemma.
\begin{lemma}\label{firstlemma}
Let $A$ be a fractal set defined by similitudes $\{\psi_m\}_{m=1}^M$ with contraction ratios $\{r_m\}_{m=1}^M$. There exists a constant $C>0$ such that for any positive integers $N_1, \ldots, N_M$ we have
$$
\EE_s(A, N_1+\cdots+N_M) \leqslant \sli_{m=1}^M r_m^{-s} \EE_s(A, N_m) + C(N_1+\cdots+N_M)^2.
$$
\end{lemma}
\begin{proof}
For every $m$, let $\widetilde\omega_{N_m}$ be a configuration optimal for $\EE_s(A, N_m)$. Define $\omega_{N_m}:=\psi_m(\widetilde\omega_{N_m})$, and set
$$
N:=N_1+\cdots+N_m, \;\; \mbox{and} \;\; \omega_N:=\omega_{N_1} \cup \cdots \cup \omega_{N_m}.
$$
We recall that the sets $\psi_m(A)$ are disjoint, thus $\#\omega_N = N$, and
$$
\EE_s(A, N)\leqslant E_s(\omega_N).
$$
Notice that
$$
E_s(\omega_N) = \sli_{i\not = j} \frac1{|x_i-x_j|^{s}} =\sli_{m=1}^M \sli_{\stackrel{i\not=j}{x_i, x_j\in \omega_{N_m}}} \frac1{|x_i-x_j|^{s}} + \sli_{\stackrel{m, m' =1\ldots M}{m\not=m'}} \sli_{\stackrel{x_i\in \omega_{N_m}}{x_j\in \omega_{N_{m'}}}} \frac1{|x_i-x_j|^{s}} =: I + II.
$$
We start with the first sum (denoted by $I$). Since $x_i, x_j\in \omega_{N_m}$, there exist $\widetilde{x_i}, \widetilde{x_j}\in \widetilde\omega_{N_m}$ such that $x_i = \psi_m(\widetilde{x_i})$ and $x_j=\psi_m(\widetilde{x_j})$. Therefore,
$$
|x_i-x_j|=|\psi_m(\widetilde{x_i}) - \psi_m(\widetilde{x_j})| = r_m |\widetilde{x_i}-\widetilde{x_j}|, 
$$
and for a fixed $m$ we have
$$
\sli_{\stackrel{i\not=j}{x_i, x_j\in \omega_{N_m}}} \frac1{|x_i-x_j|^{s}} = r_m^{-s} \sli_{\stackrel{i\not=j}{\widetilde x_i, \widetilde x_j\in \widetilde \omega_{N_m}}} \frac1{|\widetilde x_i-\widetilde x_j|^{s}} = r_m^{-s} E_s(\widetilde \omega_{N_m}) = r_m^{-s} \EE_s(A, N_m).
$$
This implies 
$$
I = \sli_{m=1}^M r_m^{-s} \EE_s(A, N_m),
$$
and to complete the proof we need to show the estimate
$$
II \leqslant C N^2.
$$
Indeed, since the sets $A_m=\psi_m(A)$ are disjoint and compact, there exists a $\sigma>0$ such that, if $x\in A_m$ and $y\in A_{m'}$ (for $m\not=m'$), then $|x-y|\geqslant \sigma$. Therefore,
$$
II \leqslant \sigma^{-s} \sli_{\stackrel{m, m' =1\ldots M}{m\not=m'}} N_m \cdot N_{m'} \leqslant C N^2,
$$
since $N_m \leqslant N$ for every $m\in \{1, \ldots , M\}$.
\end{proof}
\begin{corollary}
Let $A$ be a fractal set defined by similitudes $\{\psi_m\}_{m=1}^M$ with contraction ratios $\{r_m\}_{m=1}^M$. Let $d$ be the Hausdorff dimension of $A$, and $s>d$. There exists a positive number $C$ such that for every $N\geqslant 2$ we have
\begin{equation}\label{leia}
\EE_s(A, N) \leqslant \sli_{m=1}^M r_m^{-s} \EE_s(A, \lfloor r_m^d N \rfloor) + C(N^2 + N^{s/d}).
\end{equation}
\end{corollary}
\begin{proof}
Recall that, the Hausdorff dimension $d$ is defined by the equality
$$
\sli_{m=1}^M r_m^d = 1.
$$
Set $N_m:=\lfloor r_m^d N \rfloor + 1$. Then $N_m \geqslant r_m^d N$, thus
$$
N = \sli_{m=1}^M {r_m^d N} \leqslant \sli_{m=1}^M N_m.
$$
Therefore, Lemma \ref{firstlemma} implies
\begin{equation}\label{tili}
\EE_s(A, N) \leqslant \EE_s(A, N_1+\cdots+N_M) \leqslant \sli_{m=1}^M r_m^{-s} \EE_s(A, N_m) + C_1(N_1+\cdots+N_M)^2.
\end{equation}
First note that 
$$
\sli_{m=1}^M N_m \leqslant N+M \leqslant N\cdot M,
$$
thus 
\begin{equation}\label{trali}
C_1(N_1+\cdots+N_M)^2 \leqslant C_2 N^2. 
\end{equation}
Furthermore, 
$$
\EE_s(A, N_m) = \EE_s(A, \lfloor r_m^d N \rfloor +1).
$$
Let $\omega_{\lfloor r_m^d N\rfloor}$ be a configuration optimal for $\EE_s(A, \lfloor r_m^d N \rfloor)$. Take any $y\in A$ and denote
$$
\omega_{\lfloor r_m^d N\rfloor+1}:=\omega_{\lfloor r_m^d N\rfloor} \cup \{y\}.
$$
Then
$$
\EE_s(A, \lfloor r_m^d N \rfloor +1) \leqslant E_s(\omega_{\lfloor r_m^d N\rfloor+1}) = E_s(\omega_{\lfloor r_m^d N\rfloor}) + 2\sli_{x\in \omega_{\lfloor r_m^d N\rfloor}} \frac{1}{|x-y|^{-s}} = \EE_s(A, \lfloor r_m^d N\rfloor) +2\sli_{x\in \omega_{\lfloor r_m^d N\rfloor}} \frac{1}{|x-y|^{-s}}.
$$
Since this is true for any $y\in A$, we obtain
\begin{equation}\label{vali}
\EE_s(A, \lfloor r_m^d N \rfloor +1) \leqslant \EE_s(A, \lfloor r_m^d N\rfloor) + 2\inf_{y\in A}\sli_{x\in \omega_{\lfloor r_m^d N\rfloor}} \frac{1}{|x-y|^{-s}}. 
\end{equation}
It follows from $d$-regularity of $A$ (see equation \eqref{eq:dregular}) that
\begin{equation}\label{etomi}
2\inf_{y\in A}\sli_{x\in \omega_{\lfloor r_m^d N\rfloor}} \frac{1}{|x-y|^{-s}} \leqslant CN^{s/d}.
\end{equation}
For the complete proof of this one can look, for example, in \cite{Erdelyi2013}. To finish the proof we plug the estimates \eqref{trali}, \eqref{vali}, \eqref{etomi} into \eqref{tili}.
\end{proof}

We are now ready to prove Theorem \ref{main1}.
\begin{proof}[Proof of Theorem \ref{main1}]
Denote
$$
z_n:=r^{n(s+d)}\cdot\EE_s(A, \lfloor \ell r^{-nd} \rfloor).
$$

Our goal is to prove that the limit of $z_n$, as $n\to \infty$, exists. We first note that since the set $A$ is $d$-regular (see \eqref{eq:dregular}), we have $z_n \leqslant C_1$ for some constant $C_1$ that does not depend on $n$. This follows, for example, from \cite{Erdelyi2013}, and can also be found in \cite{saffbook}.

Since $\ell$ is fixed, from \eqref{leia} we obtain (for $C$ dependent on $\ell$, but not on $n$)
$$
\EE_s(A, \lfloor \ell r^{-nd} \rfloor) \leqslant \sli_{m=1}^M r_m^{-s}\EE_s(A, \lfloor \ell r_m^d r^{-nd} \rfloor) + C (r^{-2nd} + r^{-ns}) = \sli_{m=1}^M r^{-i_m s}\EE_s(A, \lfloor \ell r^{(i_m-n)d} \rfloor) + C (r^{-2nd} + r^{-ns}).
$$
Multiplying both sides by $r^{n(s+d)}$, we get
$$
z_n \leqslant \sli_{m=1}^M r_m^d z_{n-i_m} + C (r^{n(s-d)} + r^{nd}).
$$
We further denote $f_{i_m}:=r_m^d$, and $f_j=0$ if $j\not \in \{i_m\}_{m=1}^M$. As in the Theorem \ref{threnew}, set
$$
b_n:=z_n-\sli_{k=0}^n f_k z_{n-k}.
$$
To apply Theorem \ref{threnew} we need to show that the series $\sum |b_n|$ converges. We can obviously start with $n>\max(i_m)$, so 
\begin{equation}\label{anotherestimate}
b_n \leqslant C (r^{n(s-d)} + r^{nd}).
\end{equation}
We now utilize the telescopic nature of the partial sum $\sum_{n=0}^L b_n$. Indeed, we first compute for any $L\geqslant 0$:
\begin{equation}
\sli_{n=0}^L \sli_{k=0}^n f_k z_{n-k}=\sli_{n=0}^L \sli_{k=0}^{n} f_{n-k}z_k =\\
\sli_{k=0}^L \sli_{n=k}^L f_{n-k}z_k = \sli_{k=0}^L \left(z_k \sli_{n=k}^L f_{n-k}\right) = \sli_{k=0}^L \left(z_k \sli_{n=0}^{L-k} f_n\right).
\end{equation}
Notice that when $L-k\geqslant\max(i_m)$, we have 
$$
\sli_{n=0}^{L-k} f_n = 1. 
$$
Therefore,
$$
\sli_{n=0}^L \sli_{k=0}^n f_k z_{n-k} = \sli_{k=0}^L \left(z_k \sli_{n=0}^{L-k} f_n\right) = \sli_{k=0}^{L-\max(i_m)} z_k + \sli_{k=L-\max(i_m)+1}^L \left(z_k \sli_{n=0}^{L-k} f_n\right).
$$
We finally plug this into the following formula:
$$
\sli_{n=0}^L b_n = \sli_{n=0}^L z_n - \sli_{n=0}^L \sli_{k=0}^n f_k z_{n-k}  = \sli_{k=L-\max(i_m)+1}^L \left(z_k \cdot\left(1- \sli_{n=0}^{L-k} f_n\right)\right) = \sli_{k=L-\max(i_m)+1}^L \left(z_k \cdot\sli_{n=L-k+1}^{\max(i_m)} f_n\right).
$$
Recall that in the beginning of the proof we observed that $z_k\leqslant C_1$ for some $C_1>0$ that does not depend on $k$; moreover, since the number of terms in the outer sum does not depend on $L$, we get
\begin{equation}\label{freeze}
\left| \sli_{n=0}^L b_n  \right| \leqslant C_2. 
\end{equation}
Denote $b_n^+ = \max(b_n, 0)$ and $b_n^- = \max(-b_n, 0)$. The estimate \eqref{anotherestimate} implies
$$
\sli_{n=0}^\infty b_n^+ <\infty,
$$
and this together with \eqref{freeze} implies 
$$
\sli_{n=0}^\infty b_n^- < \infty,
$$
since otherwise we would have
$$
\sli_{n=0}^L b_n = \sli_{n=0}^L b_n^+ - \sli_{n=0}^L b_n^- \to -\infty, \;\;\; \mbox{as} \;\;\; L\to\infty.
$$
Therefore,
$$
\sli_{n=0}^\infty |b_n|<\infty,
$$
which completes the proof.
\end{proof}
\section{Proof of Theorem \ref{main2}}
We now prove Theorem \ref{main2}; i.e., the non-existence of the limit for dependent contraction ratios. Recall that we want to prove that the limit
$$
\lim_{N\to \infty} \delta(A, N)N^{1/d} 
$$
does not exist. To do so, we define the distribution function of $\delta(A, N)$; namely, set
\begin{equation}\label{distrfunction}
N(t):=\max\{N\colon \delta(A, N)\geqslant t\}.
\end{equation}
The following proposition proves Theorem \ref{main2}.
\begin{lemma}\label{lastlemma}
Let $A$ be as in Theorem \ref{main2} and set $R_n:=N(r^n)$. Then, for some $C<1$ and sufficiently large $n$, we have
$$
\delta(A, R_n+1) \leqslant Cr^n.
$$
\end{lemma}
We first show how this lemma implies the main result
\begin{proof}[Proof of Theorem \ref{main2}]
It is clear that $R_n\to \infty$ as $n\to \infty$. Since $\delta(A, R_n)\geqslant r^n$, we have
$$
\limsup_{N\to \infty}\delta(A, N)N^{1/d} \geqslant \limsup_{n\to \infty} r^n R_n^{1/d}.
$$
On the other hand,
$$
\liminf_{N\to \infty}\delta(A, N)N^{1/d} \leqslant \liminf_{n\to \infty} \delta(A, R_n+1)(R_n+1)^{1/d} \leqslant C \liminf_{n\to \infty} r^n (R_n+1)^{1/d} \leqslant C \limsup_{n\to  \infty}r^n R_n^{1/d},
$$
where in the last estimate we used that $R_n\to\infty$, and so $(R_n+1)/R_n \to 1$ as $n\to \infty$. Since $C<1$, we obtain
$$
\liminf_{N\to \infty}\delta(A, N)N^{1/d} < \limsup_{N\to \infty}\delta(A, N)N^{1/d},
$$
which completes the proof.
\end{proof}
\begin{proof}[Proof of Lemma \ref{lastlemma}]
Recall that the sets $A_m=\psi_m(A)$ are disjoint and compact, so there exists a $\sigma>0$ such that, if $x\in A_m$ and $y\in A_{m'}$ (for $m\not=m'$), then $|x-y|\geqslant \sigma$. Set 
$$
L:=\min\{k\in \mathbb{N}\colon r^k < \sigma\},
$$
and $J:=\max(L, i_1, \ldots, i_M)$. We first observe that if $n\geqslant J$, then
\begin{equation}\label{igogo}
R_n = \sli_{m=1}^M R_{n-i_m}.
\end{equation}
The proof of this relation is similar to the proof of Lemma \ref{firstlemma} and is based on the following: assume for every $m=1,\ldots, M$ we have a configuration $\widetilde\omega_{R_{n-i_m}} = \{\widetilde x_1, \ldots, \widetilde x_{R_{n-i_m}}\}\subset A$ that satisfies $\min_{i\not=j}|\widetilde x_i - \widetilde x_j|\geqslant r^{n-i_m}$. Then the configuration $\omega_{R_{n-i_m}}:=\psi_m(\widetilde\omega_{R_{n-i_m}})=\{x_1, \ldots, x_{R_{n-i_m}}\}$ satisfies $\min_{i\not=j}| x_i -  x_j|\geqslant r^{n}$. 
We leave the details as an exercise for the reader. 

Without loss of generality, assume $i_1 = \max_m i_m$, and define
$$
C:=\max_{0\leqslant n \leqslant i_1} r^{-J-n} \delta(A, R_{J+n}+1).
$$
Recall that $R_{J+n}=N(r^{J+n})$, and by definition of $N(t)$ we have $\delta(A, R_{J+n}+1) < r^{J+n}$. Since the maximum is taken over a finite set, we obtain $C<1$. It remains to prove that for every $n\geqslant J $ we have
$$
\delta(A, R_n+1)\leqslant C r^{n}.
$$
We proceed by induction. By definition of $C$, the desired estimate holds for $n=J, \ldots, J+i_1$. Assume this estimate holds for every $n=J, \ldots, n_0-1$; the induction step requires to prove it for $n=n_0$. To do this, take a configuration $\omega_{R_{n_0}+1}$ that is optimal for $\delta(A, R_{n_0}+1)$. Since $n_0>J$, we can apply \eqref{igogo}:
\begin{equation}\label{mooo}
R_{n_0} + 1 = \sli_{m=1}^M R_{n_0-i_m} + 1,
\end{equation}
Recall that $A_m=\psi_m(A)$. Assume for every $m=1,\ldots, M$ we have
$$
\#(\omega_{R_{n_0}+1} \cap A_m) \leqslant R_{n_0-i_m}.
$$
Then
$$
R_{n_0}+1 = \sli_{m=1}^M \#(\omega_{R_{n_0}+1} \cap A_m) \leqslant \sli_{m=1}^M R_{n_0-i_m},
$$
which contradicts \eqref{mooo}. Thus, for some $m$ we have $\#(\omega_{R_{n_0}+1} \cap A_m) > R_{n_0-i_m}$, and since both sides are integers, we obtain
$$
\#(\omega_{R_{n_0}+1} \cap A_m) \geqslant R_{n_0-i_m}+1.
$$
Since the function $\delta(A, \cdot)$ is decreasing, we get
\begin{multline}
\delta(A, R_{n_0}+1) =\min_{\stackrel{i\not=j}{x_i, x_j \in \omega_{R_{n_0}+1}}} |x_i-x_j| \leqslant
\min_{\stackrel{i\not=j}{x_i, x_j \in \omega_{R_{n_0}+1}\cap A_m}} |x_i-x_j|  \leqslant \delta(A_m, \#(\omega_{R_{n_0}+1} \cap A_m))  \\
\leqslant\delta(A_m, R_{n_0-i_m}+1) = r^{i_m} \delta(A, R_{n_0-i_m}+1).
\end{multline}
To finish the proof, note that $J\leqslant n_0-i_m \leqslant n_0-1$, and by the induction step
$$
\delta(A, R_{n_0}+1) \leqslant r^{i_m} \delta(A, R_{n_0-i_m}+1) \leqslant C r^{i_m} r^{n_0-i_m} = Cr^{n_0}.
$$
\end{proof}

\section{An example}\label{Sec:example}
In this section we carry out the details for Example \ref{example}. Recall that the set $A$ is defined by similitudes
$$
\psi_1(x)=x/4, \;\;\; \psi_2(x)=x/2+1/2
$$
acting on $[0,1]$. It is clear that $\delta(A, 2)=1$. We notice that $\delta(A, 3)=1/2$ with the optimal configuration $\omega_3 = \{0, 1/2, 1\}$. Our goal is to prove that 
$$
\delta(A, N) = 2^{3-n}, \;\;\;\; N\in (F_{n-1}, F_n].
$$
We remark that, in terms of the distribution function \eqref{distrfunction}, we need
$$
N(2^{3-n})=F_n.
$$
We proceed by induction: assume for every $k=3,\ldots,n$ we know that
$$
\delta(A, F_k)=2^{3-k}.
$$
We need to prove the equality for $k=n+1$. For every $k$, let $\omega_{F_k}$ be an optimal configuration for $\delta(A, F_k)$, where $F_k$ is the $k$'th Fibonacci number. Define
$$
\widetilde{\omega}_{F_{n+1}}:=\psi_1(\omega_{F_{n-1}}) \cup \psi_2(\omega_{F_{n}}).
$$
Then
$$
\delta(A, F_{n+1})\geqslant \min_{\stackrel{x\not=y}{x,y\in \widetilde{\omega}_{F_{n+1}}}} |x-y| = \min \left(\frac{\delta(A, F_{n-1})}4, \frac{\delta(A, F_{n})}2 \right) = 2^{2-n}. 
$$
Conversely, we have either $\#(\omega_{F_{n+1}}\cap \psi_1(A))\geqslant F_{n-1}$ or $\#(\omega_{F_{n+1}}\cap \psi_2(A))\geqslant F_{n}$, which implies
$$
\delta(A, F_{n+1})\leqslant 2^{2-n}.
$$
Finally, take any $N\in (F_{n}, F_{n+1}]$ and a configuration $\omega_N$ optimal for $\delta(A, N)$. We proceed by induction again: that is, assume for every $k=3,\ldots, N$ we know that, if $k\in (F_n, F_{n+1}]$, then 
$$
\delta(A, k) = 2^{2-n}.
$$
We need to prove the equality for $k=N+1$. If $N+1\in (F_n, F_{n+1}]$, then
$$
\delta(A, N)\geqslant \delta(A, F_{n+1}) = 2^{2-n}.
$$
On the other hand, take $\omega_{N+1}$ to be a configuration optimal for $\delta(A, N+1)$. For $j=1,2$ denote
$$
N_j:=\#(\omega_N \cap \psi_j(A)).
$$
We notice that $N_j\geqslant 1$, and so $N_j<N$ for $j=1,2$. Since $N_1 + N_2 > F_n$, either $N_1>F_{n-2}$ or $N_2>F_{n-1}$. Assume first $N_1>F_{n-2}$. Then $N_1 \in (F_{n-2}, F_{n-1}]$ or $N_1\in (F_{n-1}, F_n]$. In both cases, using the induction hypothesis, we get
$$
\delta(A, N+1) \leqslant 1/4 \cdot \delta(A, N_1) \leqslant 1/4 \cdot 2^{4-n} = 2^{2-n}.
$$
The case $N_2>F_{n-1}$ can be treated similarly, and our proof is done.

\bibliographystyle{acm}
\bibliography{refs}
\end{document}